\documentclass[reqno]{amsart}
\usepackage{amssymb}
\usepackage{amsmath}
\usepackage[mathscr]{euscript}
\usepackage{times}
\usepackage{newcent}       

\usepackage{color}

\makeatletter
\@addtoreset{equation}{section}
\makeatother

\newtheorem{theorem}{Theorem}[section]
\newtheorem{lemma}[theorem]{Lemma}
\newtheorem{proposition}[theorem]{Proposition}
\newtheorem{corollary}[theorem]{Corollary}

\newtheorem{remark}[theorem]{Remark}

\newcommand{\e}{\varepsilon}

\newcommand{\lan}{\langle}
\newcommand{\ran}{\rangle}

\newcommand{\R}{\mathbb{R}}
\newcommand{\E}{\mathbb{E}}
\newcommand{\N}{\mathbb{N}}
\newcommand{\Z}{\mathbb{Z}}

\newcommand{\T}{\mathbb{T}}

\newcommand{\ex}{{\rm ex}}
\newcommand{\G}{{\rm G}}
\newcommand{\bd}{{\rm bd}}
\renewcommand{\ss}{{\rm ss}}
\renewcommand{\d}{{\rm d}}

\newcommand{\mc}[1]{{\mathcal #1}}
\newcommand{\mf}[1]{{\mathfrak #1}}

\newcommand{\bb}[1]{{\mathbb #1}}

\newcommand{\mtt}[1]{{\mathtt #1}}

\newcommand{\<}{\langle}
\renewcommand{\>}{\rangle}
\renewcommand{\tilde}{\widetilde}

\definecolor{bblue}{rgb}{.2,0.2,.8}

\begin{document}

\title[Critical stationary fluctuations]{Critical stationary
fluctuations in reaction--diffusion processes} 

\author[L. Cardoso]{Luis Cardoso}
\address{IMPA, Estrada Dona Castorina 110, CEP 22460 Rio de Janeiro, Brasil}
\email{luis.cardoso@impa.br}

\author[C. Landim]{Claudio Landim}
\address{IMPA, Estrada Dona Castorina 110, CEP 22460 Rio de Janeiro, Brasil and CNRS UMR 6085, Universit\'{e} de Rouen, France}
\email{landim@impa.br}

\author[K. Tsunoda]{Kenkichi Tsunoda}
\address{Faculty of Mathematics, Kyushu University, 744, Motooka, Nishi-ku, Fukuoka, 819-0395, Japan}
\email{ktsunoda@math.kyushu-u.ac.jp}


\date{\today}

\begin{abstract}

We study stationary fluctuations at criticality for a one-dimensional
reaction--diffusion process combining symmetric simple exclusion
dynamics with Glauber-type spin flips. The strength of the Glauber
interaction is tuned to the critical regime in which the quadratic
term in the effective potential vanishes. Focusing on the stationary
distribution, we show that the total magnetization scaled by $n^{3/4}$
exhibits non-Gaussian fluctuations.  More precisely, we prove that
under the invariant measure the rescaled magnetization converges in
distribution to a random variable with density proportional to
$\exp\{-2(\theta y^2 + y^4/2)\}$. In contrast with the previous
result, we show that the density field acting on the faster modes,
that is, those associated to zero-mean test functions, have much
smaller Gaussian fluctuations. It follows from the previous two
results that the rescaled density field projects onto the
magnetization in the sense that its action on zero-mean test functions
vanishes in the limit.
\end{abstract}

\maketitle

\section{Introduction}\label{sec:intro}

Understanding the nature of fluctuations in interacting particle
systems near criticality is a central problem in probability theory
and statistical mechanics. At critical points, standard central limit
theorems typically fail, and macroscopic observables exhibit
non-Gaussian behavior governed by universal scaling laws. A major goal
of the field is to rigorously derive such fluctuation limits from
microscopic dynamics and to identify the mechanisms that determine
their universal features.

This program was first carried out for anharmonic oscillators in a
two-well potential with attractive mean-field interaction
\cite{daw83}, and later for the Ising model with a mesoscopic Kac
interaction under Glauber dynamics in dimension $d \le 3$
\cite{bprs93, mw17, gmw25}. Deriving non-Gaussian fluctuations at
criticality for short-range interacting particle systems remains an
open problem. There are, however, results showing that a dynamical
slowdown occurs near the critical point for non-conservative
finite-range interactions \cite{ls12, bd24}.

The study of equilibrium and non-equilibrium fluctuations of
interacting particle systems has a long history, beginning with the
seminal works of Rost \cite{r81} and of Brox and Rost \cite{br84}. We
do not attempt to review the development of the theory here, and
instead refer to \cite{kl, DL} for comprehensive accounts of the
literature. We emphasize, however, that recent progress by Jara and
Menezes \cite{jm18, jm20} on nonequilibrium fluctuations for
finite-range dynamics, together with the advances in regularity
structures for discrete models developed by Erhard and Hairer
\cite{eh19} and by Huang, Matetski, and Weber \cite{hmw25}, bring us
closer to proving that the density fluctuation field at criticality
for the reaction--diffusion models considered here are described by the
$\Phi^4_d$ field for $d \le 3$. Establishing this result remains one
of the central challenges in the area. 

The purpose of this paper is to provide a rigorous characterization of
stationary critical fluctuations for a one-dimensional reaction--diffusion
model combining symmetric simple exclusion dynamics with Glauber-type spin
flips. The Glauber interaction strength is tuned at criticality, leading to a
degeneracy of the quadratic term in the effective potential and the emergence
of quartic fluctuations. We focus on the stationary distribution of the
rescaled total magnetization and show that it converges to a non-Gaussian
limit described by a quartic potential.

More precisely, we prove that under the stationary measure, the rescaled
magnetization converges weakly to the unique invariant distribution of a
one-dimensional diffusion with drift given by the derivative of a quartic
potential. This result identifies the stationary fluctuation law explicitly
and shows that it coincides with the ergodic measure of the limiting critical
diffusion obtained in the dynamical scaling limit.

The second main result of the article establishes that under the
stationary state the faster modes, that is, those associated to
zero-mean test functions, have much smaller Gaussian fluctuations. It
follows from this result and the weak convergence of the rescaled
magnetization that the rescaled density field projects onto the
magnetization in the sense that its action on zero-mean test functions
vanishes in the limit.

To the best of our knowledge, this provides the first rigorous derivation of a
non-Gaussian critical fluctuation for the stationary state of a short-range
interacting particle system. While non-Gaussian limits of this type are well
known in mean-field models such as the Curie--Weiss model, their appearance in
spatially interacting systems with local dynamics has remained largely open.
Our result shows that such critical non-Gaussian fluctuations also arise in the
stationary state of a short-range interacting particle system.

From a methodological viewpoint, the main difficulty lies in
controlling the stationary measures at criticality. While the
dynamical convergence of the magnetization process was established in
\cite{DL}, passing to the stationary limit requires uniform entropy
and moment bounds at the critical scale. A key ingredient of
our approach is a logarithmic Sobolev inequality for a sequence of
non-product reference measures that incorporate the critical tilt of
the invariant state. This allows us to prove the tightness of the
stationary fluctuations and to identify all possible limit points.

We finally mention closely related previous works.
The fluctuations under the stationary state for
the boundary-driven simple exclusion processes
were studied in \cite{lmo08, fgn19, gjnm20}.
Regarding the reaction--diffusion process,
the case where the Glauber jump rate is given by
a constant plus a small perturbation
was investigated in \cite{GJMM}. In our main theorem,
we also assume that the Glauber jump rate is small,
and a similar assumption is also imposed in \cite{GJMM}.
In any of these works, the diffusive time scale is considered,
and the limiting fluctuations are therefore naturally Gaussian.

This paper is organized as follows.  In Section \ref{sec:model}, we
introduce the reaction--diffusion process and state the main results,
Theorems \ref{main} and \ref{thm2}, in a precise manner.  The proof of
these results is presented in Section \ref{sec:pfofmain}, whereas the
proof of the aforementioned logarithmic Sobolev inequality is
postponed to the latter sections.  In Section \ref{sec:lsi}, we prove
the logarithmic Sobolev inequality based on one for some auxiliary
birth-death process.  Technical computations for the birth-death
process are provided in Section \ref{sec:pflogsobolev} in detail.

\section{Model and result}\label{sec:model}

For each $n\in\N$, let $\T_n$ be the one-dimensional discrete torus
$\Z/n\Z$ and $\Omega_n$ be the configuration space $\{0,1\}^{\T_n}$.
We denote its generic element by $\eta=(\eta(x))_{x\in\T_n}$.
We define a Markovian operator $L_n:=n^2L_n^\ex+aL_n^\G$, where
$L_n^\ex$ and $L_n^\G$ represent the generator of the symmetric simple exclusion process and the Glauber dynamics, respectively:
\begin{align*}
L_n^\ex f(\eta)&=\sum_{x\in\T_n}\left\{f(\eta^{x,x+1}) - f(\eta)\right\},\\
L_n^\G f(\eta)&=\sum_{x\in\T_n}c_x(\eta)\left\{f(\eta^{x}) - f(\eta)\right\}.
\end{align*}
In these formulas, $\eta^{x,x+1}$ and $\eta^x$ represent
the configuration obtained from $\eta$ by exchanging
the occupation variables at $x$ and $x+1$ and
by flipping the value at $x$, respectively:
\begin{align*}
\eta^{x,x+1}(y)=
\begin{cases}
\eta(x+1), & y=x,\\
\eta(x), & y=x+1,\\
\eta(y), & y\neq x, x+1,
\end{cases}
\quad 
\eta^{x}(y)=
\begin{cases}
1-\eta(x), & y=x,\\
\eta(y), & y\neq x.
\end{cases}
\end{align*}
The Glauber jump rate $c_x(\eta)$ is defined as
\begin{align*}
c_x(\eta)=1-\gamma\sigma(x)(\sigma(x-1)+\sigma(x+1)) +\gamma^2\sigma(x-1)\sigma(x+1),
\end{align*}
for some $0\le \gamma \le 1$, where $\sigma(x):=2\eta(x)-1$.
In this paper, we consider the case where $\gamma$ varies 
as a function in $n$:
\begin{align*}
\gamma=\gamma_n=\dfrac12\left(1-\dfrac{\theta}{\sqrt n}\right),
\end{align*}
for fixed $\theta\in\R$.
The constant $a>0$ is also given and fixed, and is chosen to be small enough later.

Since the reaction--diffusion process is an irreducible Markov process
on the finite set $\Omega_n$, there uniquely exists a probability measure
on $\Omega_n$, denoted by $\mu_\ss^n$, which is invariant under the dynamics of $L_n$.
Let $\mc Y^n:\Omega_n\to\R$ be the rescaled total magnetization:
\begin{align*}
\mc Y^n=\mc Y^n(\eta):=\dfrac{1}{n^{3/4}}\sum_{x\in\T_n}(\eta(x)-1/2).
\end{align*}
For each $H\in C^\infty(\T)$, we also define
\begin{align*}
\mtt Y^n(H)=\mtt Y^n(H;\eta):=\dfrac{1}{n^{3/4}}\sum_{x\in\T_n}H(x/n)(\eta(x)-1/2), \quad
\lan H \ran :=\int_{\T} H(y) dy.
\end{align*}
Let $\alpha^n$ be the distribution of $\mc Y^n$ under $\mu_\ss^n$,
that is,
\begin{align*}
\alpha^n(A):=\mu_\ss^n\left(\eta\in\Omega_n: \mc Y^n(\eta)\in A \right), \quad A\in\mc B(\R).
\end{align*}
Define the function $V(y)=\theta y^2 + y^4/2, y\in\R$
and the measure $\alpha^*$ on $\R$ by
\begin{align*}
\alpha^*(dy)=Z^{-1}e^{-2V(y)}dy, \quad Z=\int_{\R} e^{-2V(y)}dy,
\end{align*}
where $dy$ stands for the one-dimensional Lebesgue measure on $\R$.
Moreover, let $\mtt Y^*$ be an $\R$-valued random variable having the distribution $\alpha^*$.

The following theorem is the main result of this paper
and is proved in Section \ref{sec:pfofmain}.

\begin{theorem}\label{main}
For sufficiently small $a>0$ the sequence of the probability measures
$\{\alpha^n\}_{n\in\N}$ weakly converges to the measure $\alpha^*$ as
$n\to\infty$.
\end{theorem}

\begin{remark}
To prove Theorem \ref{main}, it suffices to take $a$ sufficiently small so that
(1) the coefficient of the Dirichlet form term in \eqref{ineq20} is negative,
and (2) \cite[(i) of Theorem 2.3]{DL} is valid.
For the assertion (1), see the proof of Lemma \ref{lem4}.
Note also that the assertion (2) is used in the proof of Lemma \ref{lem5}.
\end{remark}

The second main result of the article asserts that the faster modes,
that is, those associated to zero-mean test functions, have much
smaller Gaussian fluctuations. This is the content of the next result.
Denote by $y^n$ the density field with Gaussian scaling: for any
smooth $H\colon \bb T \to\bb R$,
\begin{equation}
y^n (H) \,:=\,
\frac{1}{\sqrt{n}\,} \sum_{x\in\T_n} H(x/n)\, [\, \eta (x) - 1/2\,] \,.
\end{equation}
Denote by $\<\,\cdot\,,\,\cdot\>$ the scalar product in $\bb L^2(\T)$. 

\begin{theorem}
\label{thm2}
Assume that the hypotheses of Theorem \ref{main} are in force. Then,
under the measure $\mu_{\ss}^n$, the field $(y^n)_{n\ge 1}$ weakly
converges to a Gaussian field $y$ on $\{H\in \bb L^2(\T) : \<H\>=0\}$,
with covariance:
\begin{equation}
\E \big[\,y(H)\, y (G) \, \big]
\,=\, 
\frac{1}{4} \<H,G\> + \frac{a}{2}\big\<H,(-\Delta)^{-1} G\big\> \,.
\end{equation}
\end{theorem}

It follows from the two previous results that under the measure 
$\mu_{\ss}^n$ the rescaled density field $\mtt Y^n(\cdot) $ projects
onto the magnetization. In particular, its action on zero-mean test
functions vanishes in the limit.

\begin{corollary}
\label{cor1}
For any $H\in C^\infty(\T)$, the sequence of the random variables
$\{\mtt Y^n(H)\}_{n\in\N}$ under $\mu_{\ss}^n$ weakly converges to the
random variable $\lan H\ran \mtt Y^*$ as $n\to\infty$.
\end{corollary}

\section{Proof of Theorems \ref{main} and
\ref{thm2}}\label{sec:pfofmain}

In this section, we prove the main results of the article.  As our
proof relies on several results established in \cite{DL}, we start by
introducing the necessary notation used there.

Define $U:[-1/2,1/2]\to\R$ by
\begin{align*}
U(\varrho):=\dfrac{1}{\gamma}\left\{[1+2\gamma\varrho]\log[1+2\gamma\varrho] + [1-2\gamma\varrho]\log[1-2\gamma\varrho] \right\}, \quad \varrho\in[-1/2,1/2],
\end{align*}
and $U_0:[0,1]\to\R$ by $U_0(\rho):=U(\rho-1/2), \rho\in[0,1]$.
The relationship between $U_0$ and the birth and death jump rates $ B$ and $ D$ is as follows.
Recall that
\begin{align*}
B(\rho)=(1-\rho)[1+\gamma(2\rho-1)]^2, \quad D(\rho)=\rho[1-\gamma(2\rho-1)]^2,
\end{align*}
for $\rho\in[0,1]$. Then
\begin{align*}
\log B(\rho) -\log D(\rho)= E'(\rho) + U_0'(\rho),
\end{align*}
where
\begin{align*}
E(\rho):=\rho\log\rho+(1-\rho)\log(1-\rho)+\log 2.
\end{align*}
Let $W(\rho) := E(\rho) - U_0(\rho), \rho\in[0,1]$.  When
$\gamma = \gamma_n = \frac{1}{2} \big ( 1-\theta n^{-1/2}\big)$, that
$W$ satisfies the following conditions: 
\begin{equation}
W(\frac{1}{2}) = W'(\frac{1}{2}) = W^{(3)} (\frac{1}{2}) = 0, W''(\frac{1}{2}) = \frac{4\theta}{\sqrt{n}}, W^{(4)} (\rho) \ge c_W \text{ for all $\rho\in[0,1]$}\label{W_properties}
\end{equation}
for some finite constant $c_W>0$. In this formula, $W^{(j)}(\cdot)$,
$j\ge 3$, stands for the $j$-th derivative of $W$.
Note that $W$ has a dependency on $n$.

Let $\nu_{1/2}^n$ be the uniform measure on $\Omega_n$.
We also define the measure $\nu_U^n$ by
\begin{align*}
\nu_U^n(\eta):=\dfrac{e^{nU(\overline{m_n}/n)}}{Z_U^n}\nu_{1/2}^n(\eta), \quad \eta\in\Omega_n,
\end{align*}
where
$$\overline{m_n}=\overline{m_n(\eta)}:=\sum_{x\in\T_n}(\eta(x) -1/2).$$

For each density $f:\Omega_n\to\R$ with respect to $\nu_U^n$,
we define the relative entropy and the Dirichlet forms
with respect to $\nu_U^n$ by
\begin{align*}
H_n(f|\nu_U^n)&:=\int_{\Omega_n} f\log fd\nu_U^n,\\
D_n^\ex(f;\nu_U^n)&:=\dfrac12\sum_{x\in\T_n}\int_{\Omega_n} \left\{\sqrt{f(\eta^{x,x+1})} -\sqrt{f(\eta)}\right\}^2 d\nu_U^n,\\
D_n^G(f;\nu_U^n)&:=\dfrac12\sum_{x\in\T_n}\int_{\Omega_n} c_x(\eta)\left\{\sqrt{f(\eta^{x})} -\sqrt{f(\eta)}\right\}^2 d\nu_U^n,\\
D_n(f;\nu_U^n)&:=n^2D_n^\ex(f;\nu_U^n) + D_n^G(f;\nu_U^n).
\end{align*}
Let $f_\ss^n$ be the Radon-Nikodym derivative of $\mu_\ss^n$ with respect to $\nu_U^n$.

The next result is crucial to invoke the dynamical result established in \cite{DL}.

\begin{lemma}\label{lem4}
For sufficiently small $a>0$, we have
\begin{align*}
\sup_{n\in\N} \left\{\dfrac{H_n(f_\ss^n|\nu_U^n)}{\sqrt n} + E_{\mu_\ss^n}[(\mc Y^n)^4] \right\} < \infty.
\end{align*}
In particular, the sequence of the measures $\{\alpha^n\}_{n\in\N}$ is tight.
\end{lemma}
\begin{proof}
We first consider the bound on the relative entropy.  Let
$(\eta_t^n)_{t\ge0}$ be the process generated by $\sqrt n L_n$, which
starts from $\nu_U^n$.  Denote the distribution of $\eta_t^n$ by
$\mu_t^n$ and the Radon-Nikodym derivative of $\mu_t^n$ with respect
to $\nu_U^n$ by $f_t^n$.  As shown in the proof of \cite[Theorem
2.7]{DL} (see the inequality after (4.14)) that there exists a
constant $C_1>0$ such that
\begin{align}\label{ineq20}
H_n'(f_t^n|\nu_U^n) \le(aC_1-2)\sqrt{n}D_n(f_t^n;\nu_U^n)+aC_1\sqrt{n},
\end{align}
for any $t\ge0$. If $a\le C_1^{-1}$, by Theorem \ref{thm1}
(logarithmic Sobolev inequality for $\nu_U^n$) we have
\begin{align*}
H_n'(f_t^n|\nu_U^n) \le-C^{-1}H_n(f_t^n|\nu_U^n)+\sqrt{n},
\end{align*}
where $C$ is a constant given in Theorem \ref{thm1}.
From this inequality, we obtain $H_n(f_t^n|\nu_U^n)\le C\sqrt n$ for any $t\ge0$.
Since the reaction--diffusion process is irreducible,
$f_t^n$ pointwisely converges to $f_\ss^n$ as $t\to\infty$.
Hence we obtain $H_n(f_\ss^n|\nu_U^n)\le C\sqrt n$.

We next consider the moment bound in the statement.
It is also proved in the proof of \cite[Corollary 5.4]{DL} (see (5.2)) that
there exist a constant $C_2>0$ and $N_0\in\N$ such that
\begin{align*}
&e^n_t+16a\gamma^2\sqrt{n} \bb E_{\nu_U^n}\left[\int_0^t (\mc Y^n_s)^4 ds\right]\\
&\quad \le e_0^n -4a\theta \int_0^t e_s^n ds + C_2a\{ H_n(f_0^n|\nu_U^n) +\sqrt{n}t\}, 
\end{align*}
for all $n\ge N_0$ and $t\ge0$,
where $\mc Y_t^n:=\mc Y^n(\eta_t^n)$ and $e_t^n:=\sqrt{n} \bb E_{\nu_U^n}[(\mc Y_t^n)^2]$.
We use the elementary inequality $x^4\ge4A(x^2-A)$ for
any $x\in\R$ and $A>0$ to bound the left-hand side.
Letting $x=\mc Y_s^n$ and $A=|\theta|(8\gamma^2)^{-1}$,
the left-hand side is bounded below by
\begin{align*}
8a\gamma^2\sqrt{n} \bb E_{\nu_U^n}\left[\int_0^t (\mc Y^n_s)^4 ds\right] + 4a|\theta| \int_0^t e_s^n ds - \dfrac{a|\theta|^2}{2\gamma^2}\sqrt n t,
\end{align*}
and thus
\begin{align*}
&8a\gamma^2\sqrt{n} \bb E_{\nu_U^n}\left[\int_0^t (\mc Y^n_s)^4 ds\right]  - \dfrac{a|\theta|^2}{2\gamma^2}\sqrt n t\\
&\quad \le e_0^n -4a(\theta+|\theta|) \int_0^t e_s^n ds + C_2a\{ H_n(f_0^n|\nu_U^n) +\sqrt{n}t\}, 
\end{align*}
It follows from these bounds and the fact that $f_0^n=1$ that
\begin{align*}
8a\gamma^2\sqrt{n} \bb E_{\nu_U^n}\left[\int_0^t (\mc Y^n_s)^4 ds\right]
\le e_0^n +C_3a\sqrt{n}t,
\end{align*}
for some constant $C_3>0$, which is independent of $n$ and $t$.
Letting $\tilde\mu_t^n:=t^{-1}\int_0^t \mu_s^n ds$, the last estimate can be rewritten as
\begin{align*}
8a\gamma^2\sqrt n E_{\tilde\mu_t^n}[(\mc Y^n)^4] \le \dfrac{e_0^n}{t} +C_3a\sqrt{n}.
\end{align*}
Since $\tilde\mu_t^n$ converges to $\mu_\ss^n$ as $t\to\infty$,
we obtain the desired moment bound.

The tightness of the sequence $\{\alpha^n\}_{n\in\N}$ immediately
follows from the moment bound.
\end{proof}

As we established the relative compactness of the sequence
$\{\alpha^n\}_{n\in\N}$, to prove Theorem \ref{main} it remains to
provide some characterization of limit points.  In the next lemma, we
show that any weak limit point is the ergodic measure of the diffusion
\begin{align}\label{sde}
\d\mc Y_t=-aV'(\mc Y_t)\d t +\sqrt{a} \d W_t,
\end{align}
where $(W_t)_{t\ge0}$ is a standard Brownian motion.
To state its assertion, let $\alpha P_t$ be the distribution of the diffusion \eqref{sde}
starting from a Borel probability measure $\alpha$ on $\R$
at time $t\ge0$.

\begin{lemma}\label{lem5}
Let $\tilde\alpha$ be any weak limit point of the sequence $\{\alpha^n\}_{n\in\N}$.
For sufficiently small $a>0$,
we have $\tilde\alpha=\tilde\alpha P_t$ for any $t\ge0$. In particular, $\tilde\alpha=\alpha^*$.
\end{lemma}
\begin{proof}
To prove this lemma, we use the convergence result of the magnetization process,
which is established in \cite{DL}. Let $\{\mu_n\}_{n\in\N}$ be a sequence of probability measures on $\Omega_n$ satisfying
\begin{align}\label{assDL}
H_n(\mu_n|\nu_{1/2}^n)=O(\sqrt n), \quad \lim_{n\in\N}\mu_n(\mc Y^n\in\cdot)=\alpha \quad \text{weakly},
\end{align}
for some Borel probability measure $\alpha$ on $\R$.
We denote by $(\eta_t^n)_{t\ge0}$  the process generated by $\sqrt n L_n$, which starts from $\mu_n$ and define the magnetization process by $\mc Y_t^n=\mc Y^n(\eta_t^n)$.
It has been proved in \cite[Theorem 2.3]{DL} that, for any $T>0$ and $p\in(1,4/3)$,
the law of $(\mc Y_t^n)_{t\in[0,T]}$ under $\bb P_{\mu_n}^n$,
seen as a probability measures on $\bb L^p([0,T], \R)$, converges to the measure $\bb Q_\alpha$ in the topology of measures on  $\bb L^p([0,T], \R)$.

To use this result, we should verify that the condition \eqref{assDL} for an initial distribution is in force.
However, the following condition, together with the weak convergence of
$\{\mu_n\}_{n\in\N}$ is enough to deduce the convergence of the laws of $(\mc Y_t^n)_{t\in[0,T]}$:
\begin{align}\label{sufDL}
\sup_{n\ge1}\left\{\dfrac{H_n(\mu_n|\nu_U^n)}{\sqrt n} + E_{\mu_n}[(\mc Y^n)^2] \right\} < \infty.
\end{align}
Indeed, the entropy bound in \eqref{assDL} is used only for proving
\eqref{sufDL}. This is demonstrated in \cite[Proposition 2.14]{DL}.
An actual proof of the convergence result is based on the bounds in
\eqref{sufDL}, see \cite[Section 7]{DL}.  We consider the case where
$\mu_n=\mu_\ss^n$, and in this case the condition \eqref{sufDL} was
proved in Lemma \ref{lem4}.  In conclusion, we can deduce the
sub-sequential convergence of the law of $(\mc Y_t^n)_{t\in[0, T]}$
when the process starts from the stationary state $\mu_\ss^n$.

Let $\tilde \alpha$ be any weak limit point of $\{\alpha^n\}_{n\in\N}$ and take a subsequence of $\{\alpha^n\}_{n\in\N}$ converging to $\tilde\alpha$.
For simplicity, we denote its subsequence by $\{\alpha^n\}_{n\in\N}$ again.
Since the process starts from the stationary state, for any $t>0$
the law of $\mc Y_t^n$ is given by $\alpha^n$, which converges to
$\tilde\alpha$ as $n\to\infty$. Therefore, to conclude the proof, it is enough to show that the law of $\mc Y_t^n$ also converges to $\tilde\alpha P_t$ as $n\to\infty$. This is almost obvious from
the convergence of the law of the magnetization processes.
However, one must pay attention because the convergence
holds in the topology of $\bb L^p([0, T], \R)$, not in the Skorokhod topology. 
Although the convergence at a fixed time is essentially proved in \cite{DL}, we prove it in the next paragraph for clarity.

As mentioned in the previous paragraph, we show that
for any $t>0$ the law of $\mc Y_t^n$ converges to $\tilde\alpha P_t$ as $n\to\infty$.
To prove this claim, take any $0<t<T$ and let $\iota^\e:[-\e,\e]\to[0,\infty)$ be a mollifier.
We decompose $\mc Y_t^n$ into
\begin{align}\label{eq1}
\mc Y_t^n=\int_{t-\e}^{t+\e} (\mc Y_t^n - \mc Y_s^n)\iota_{s-t}^\e ds + \int_{t-\e}^{t+\e} \mc Y_s^n\iota_{s-t}^\e ds.
\end{align}
It is proved in \cite[Subsection 7.1]{DL} that there exist $C_1>0$ and $\delta_1>0$ such that
\begin{align*}
\bb E_{\mu_n}^n\left[\left|\mc Y_t^n - \mc Y_s^n\right| \right] \le C_1(t-s)^{\delta_1},
\end{align*}
uniformly in $n\in\N$ and $0<s\le t \le T$.
This bound is not clearly stated in \cite[Subsection 7.1]{DL}
but it is a direct consequence of several estimates on the modulus of continuity.
Therefore, the first term in the right-hand side of \eqref{eq1}
vanishes in $L^1$-sense as $\e\to0$ uniformly in $n\in\N$.
On the other hand, for any fixed $\e>0$
the second term in the right-hand side of \eqref{eq1} converges
in law to $\int_{t-\e}^{t+\e} \tilde{\mc Y}_s\iota_s^\e ds$ as $n\to\infty$, where
$(\tilde{\mc Y}_u)_{u\ge0}$ is a solution to \eqref{sde}
with the initial distribution $\tilde\alpha$.
Similarly, as we have the bound
\begin{align*}
\bb E\left[\left|\mc Y_t - \mc Y_s\right| \right] \le C_2(t-s)^{\delta_2},
\end{align*}
for some $C_2>0$ and $\delta_2>0$ uniformly in $0<s\le t \le T$,
$\int_{t-\e}^{t+\e} \tilde{\mc Y}_s\iota_s^\e ds$ converges in the $L^1$-sense
to $\tilde{\mc Y}_t$ as $\e\to0$.
Thus, the law of $\mc Y_t^n$ converges to $\tilde\alpha P_t$ as $n\to\infty$,
which completes the proof.

The latter statement follows from the fact that $\alpha^*$
is ergodic with respect to the diffusion \eqref{sde}, see \cite[Section 1.5]{Berglund}.
\end{proof}

We are now in the position to prove Theorem \ref{main}.

\begin{proof}[Proof of Theorem \ref{main}]
Assume that $a$ is sufficiently small so that Lemmata \ref{lem4} and \ref{lem5}
are in force.

The statement of the theorem is an immediate consequence of Lemmata
\ref{lem4} and \ref{lem5}.  Indeed, by Lemma \ref{lem4} the sequence
$\{\alpha^n\}$ has a converging subsequence and by Lemma \ref{lem5}
its weak limit point coincides with $\alpha^*$.  Therefore the
sequence $\{\alpha^n\}_{n\in\N}$ converges to $\alpha^*$ as
$n\to\infty$.  Thus, the proof of the theorem is complete.
\end{proof}

\begin{proof}[Proof of Theorem \ref{thm2}]
The result follows from the proof of \cite[Proposition 8.1]{DL}.  The
proof of this proposition relies on the bound for the entropy of
$\mu_n$ (in our setting $\mu_{ss}^n$) with respect to $\nu_U^n$
established in Lemma \ref{lem4}. The proof is written taking as
reference measure $\nu^n_g$, but it holds as it is for $\nu^n_U$. 
\end{proof}

\begin{proof} [Proof of Corollary  \ref{cor1}]
Fix $T>0$. Let $(\eta_t^n)_{t\ge0}$ be the process generated by
$\sqrt n L_n$, which starts from $\mu_\ss^n$ and let
$\mtt Y_t^n(H)=\mtt Y^n(H;\eta_t^n)$.  Since $\mu_\ss^n$ is the
stationary state, we have
\begin{align*}
E_{\mu_\ss^n}\left[ \left| \mtt Y^n(H) - \lan H \ran \mc Y^n \right|\right] = \dfrac{1}{T}\E_{\mu_\ss^n}\left[ \int_0^T \left| \mtt Y_s^n(H) - \lan H \ran \mc Y_s^n \right| ds \right].
\end{align*}
By the reason mentioned after \eqref{sufDL}, (i) of Theorem 2.3 (see (7.3)) of \cite{DL} is valid, that is, the right-hand side of the last expression vanishes as $n\to\infty$. Therefore, we have
\begin{align*}
\lim_{n\to\infty}E_{\mu_\ss^n}\left[ \left| \mtt Y^n(H) - \lan H \ran \mc Y^n \right|\right] = 0.
\end{align*}
The conclusion follows from Theorem \ref{main}, which completes the
proof of  the corollary.
\end{proof}

\section{Logarithmic Sobolev inequality for $\nu_U^n$}\label{sec:lsi}

In this section, we prove the following logarithmic Sobolev inequality.

\begin{theorem}\label{thm1}
There exists a constant $C>0$ independent of $n$ such that
for any density $f:\Omega_n\to\R$ with respect to $\nu_U^n$
\begin{align*}
H_n(f|\nu_U^n)\le C \left( n^2D_n^\ex(f;\nu_U^n) + \sqrt{n} D_n^\G(f;\nu_U^n)\right).
\end{align*}
In particular, for any density $f:\Omega_n\to\R$ with respect to $\nu_U^n$
\begin{align*}
H_n(f|\nu_U^n)\le C \sqrt{n}D_n(f;\nu_U^n).
\end{align*}
\end{theorem}

The proof of Theorem \ref{thm1} proceeds in several steps
and provided at the end of this section.
The first step is a simple application of the logarithmic Sobolev inequality for the exclusion part.

\begin{lemma}\label{lem1}
There exists a constant $C>0$ independent of $n$ such that
for any density $f:\Omega_n\to\R$ with respect to $\nu_U^n$
\begin{align}\label{ineq1}
H_n(f|\nu_U^n) \le Cn^2 D_n^\ex(f;\nu_U^n) + \sum_{m=0}^n\lan f \ran_{\nu_{1/2}^{n,m}}\log \lan f \ran_{\nu_{1/2}^{n,m}} \nu_U^n(\Omega_{n,m}).
\end{align}
\end{lemma}

\begin{proof}
Fix a density $f:\Omega_n\to\R$ with respect to $\nu_U^n$.
We first decompose $\Omega_n$ and $\nu_{1/2}^n$ by the magnetizations.
For each $m=0,\ldots,n$, let
$$\Omega_{n,m}:=\left\{\eta\in\Omega_n \,\middle|\, \sum_{x\in\T_n}\eta(x)=m\right\}$$
and let $\nu_{1/2}^{n,m}$ be the uniform measure on $\Omega_{n,m}$.
Note that
\begin{align*}
H_n(f|\nu_U^n)
= \sum_{m=0}^n \int_{\Omega_{n,m}} f\log f d\nu_U^n
= \sum_{m=0}^n \nu_U^n(\Omega_{n,m})\int_{\Omega_{n,m}} f\log f d\nu_{1/2}^{n,m}.
\end{align*}
Letting $f^{n,m}:=f/\lan f \ran_{\nu_{1/2}^{n,m}}$ gives us
\begin{align*}
\int_{\Omega_{n,m}} f\log f d\nu_{1/2}^{n,m}
&=  \int_{\Omega_{n,m}}   f \left( \log f^{n,m} + \log \lan f \ran_{\nu_{1/2}^{n,m}} \right) d\nu_{1/2}^{n,m}\\
&= \lan f \ran_{\nu_{1/2}^{n,m}}\int_{\Omega_{n,m}}  f^{n,m} \log f^{n,m} d\nu_{1/2}^{n,m} + \lan f \ran_{\nu_{1/2}^{n,m}}\log \lan f \ran_{\nu_{1/2}^{n,m}}.
\end{align*}
Since $f^{n,m}$ is a density with respect to $\nu_{1/2}^{n,m}$,
it follows from the logarithmic Sobolev inequality
for the exclusion process (see \cite[Theorem 4]{Yau}) that
\begin{align*}
\int_{\Omega_{n,m}} f^{n,m} \log f^{n,m} d\nu_{1/2}^{n,m} \le Cn^2 D_n^\ex(f^{n,m};\nu_{1/2}^{n,m}),
\end{align*}
for some constant $C>0$, which is independent of $n$ and $m$.
Therefore we have
\begin{align*}
\sum_{m=0}^n \nu_U^n(\Omega_{n,m})
\lan f \ran_{\nu_{1/2}^{n,m}}D_n^\ex(f^{n,m};\nu_{1/2}^{n,m})
=\sum_{m=0}^n \nu_U^n(\Omega_{n,m}) D_n^\ex(f;\nu_{1/2}^{n,m})
=D_n^\ex(f;\nu_U^n).
\end{align*}
Collecting all previous estimates, we obtain \eqref{ineq1}.
\end{proof}

In view of Lemma \ref{lem1}, we need to estimate the sum appearing in \eqref{ineq1}, which is the relative entropy of $\lan f\ran_{\nu_{1/2}^{n,\cdot}}$ with respect to $\nu_U^n(\Omega_{n,\cdot})$.
To obtain estimates on this relative entropy,
we introduce an auxiliary birth-death process on $I_n:=\{0,1,\ldots,n\}$.
For each $m\in I_n$, define the jump rates $b_n(m)$ and $d_n(m)$ by
\begin{align*}
b_n(m)&:=(n-m)e^{n(U_0((m+1)/n) - U_0(m/n))/2},\\
d_n(m)&:=me^{n(U_0((m-1)/n) - U_0(m/n))/2},
\end{align*}
with the convention $b_n(n)=d_n(0):=0$.
The generator of the birth-death process is given by
\begin{align*}
\mf L_n^{\bd}\mf f(m) =b_n(m)\{\mf f(m+1) - \mf f(m)\} + d_n(m)\{\mf f(m-1) - \mf f(m)\},
\end{align*}
for functions $\mf f:I_n\to\R$.
Then the measure defined by
$\pi_n(m):=\nu_U^n(\Omega_{n,m})$ satisfies the detailed balance condition
\begin{align*}
\pi_n(m)b_n(m)=\pi_n(m+1)d_n(m+1).
\end{align*}
Therefore, $\pi_n$ is reversible with respect to $\mf L_n^\bd$.

For each density $\mf f:I_n\to\R$ with respect to $\pi_n$,
we define the relative entropy and the Dirichlet form with respect to $\pi_n$ by
\begin{align*}
\mf H_n(\mf f|\pi_n)&:=\sum_{m=0}^n \mf f(m)\log \mf f(m) \pi_n(m),\\
\mf D_n^\bd(\mf f;\pi_n)&:=\dfrac12\sum_{m=0}^{n-1}b_n(m)\left\{\sqrt{\mf f(m+1)} - \sqrt{\mf f(m)} \right\}^2 \pi_n(m).
\end{align*}

The next result establishes a logarithmic Sobolev inequality for the birth-death process.

\begin{lemma}\label{lem2}
There exists a constant $C>0$ independent of $n$ such that
for any density $\mf f:I_n\to[0,\infty)$ with respect to $\pi_n$
\begin{align*}
\mf H_n(\mf f|\pi_n) \le C\sqrt{n} \mf D_n^\bd(\mf f;\pi_n).
\end{align*}
\end{lemma}
\begin{proof}
The proof of Lemma \ref{lem2} relies on \cite[Proposition 4]{Miclo1}
and related computations provided in Section \ref{sec:pflogsobolev}.
To introduce the result \cite[Proposition 4]{Miclo1},
let $\lambda_n$ be the logarithmic Sobolev constant
\begin{align}\label{ls1}
\lambda_n:=\inf\left\{\dfrac{\mf D_n^\bd(\mf f;\pi_n)}{\mf H_n(\mf f|\pi_n)} \,\middle|\, \mf f:I_n\to[0,\infty), E_{\pi_n}[\mf f]=1\right\}.
\end{align}
For each $m=0,\ldots, n$, we also define
\begin{align*}
C_n^+(m)&:=\sup_{\ell>m} \left\{\left( \sum_{k=m+1}^\ell \dfrac{1}{\pi_n(k)d_n(k)} \right) 
\pi_n([\ell,\infty))\left|\log \pi_n([\ell,\infty)) \right| \right\},\\
C_n^-(m)&:=\sup_{\ell<m} \left\{\left( \sum_{k=\ell}^{m-1} \dfrac{1}{\pi_n(k)b_n(k)} \right) 
\pi_n((-\infty,\ell])\left|\log \pi_n((-\infty,\ell])) \right| \right\},
\end{align*}
and
\begin{align*}
C_n:=\min_{m=0,\ldots,n} (C_n^-(m) \lor C_n^+(m)).
\end{align*}
It is known that $\gamma_n$ and $C_n^{-1}$ have same order.
Indeed, the following bound is given in \cite[Proposition 4]{Miclo1}
(see also \cite[Proposition 13]{Miclo2}).
\begin{align*}
\dfrac{1}{80}\dfrac{1}{C_n}\le \lambda_n \le \dfrac{4}{3}\left(1-\dfrac{\sqrt{5}}{2\sqrt{2}} \right)^{-2} \dfrac{1}{C_n}.
\end{align*}
Letting $m_n:= n/2 $, $\lambda_n$ is bounded as
$\lambda_n\ge (80(C_n^+(m_n) \lor C_n^-(m_n)))^{-1}$.
Therefore, to conclude the proof, it is enough to show that
there exists a constant $C>0$ such that for any $n$ we have $C_n^+(m_n) \lor C_n^-(m_n)\le C\sqrt{n}$. This is the content of Lemma \ref{lem6}, which concludes the proof.
\end{proof}

The proof of the following lemma is postponed to Section \ref{sec:pflogsobolev}.

\begin{lemma}\label{lem6}
There exists a constant $C$ such that $C_n^{\pm}(m_n) \leq C\sqrt{n}$ for all $n \in \N$.
\end{lemma}

To complete the logarithmic Sobolev inequality for $\nu_U^n$, it is enough to estimate the Dirichlet forms of the birth-death process
by one of the original Glauber dynamics. This is the content of the next lemma.

\begin{lemma}\label{lem3}
There exists a constant $C>0$ independent of $n$ such that
for any density $f:\Omega_n\to\R$ with respect to $\nu_U^n$
\begin{align*}
\sum_{m=0}^{n-1}b_n(m)\left\{\sqrt{\lan f \ran_{\nu_{1/2}^{n,m+1}}} - \sqrt{\lan f \ran_{\nu_{1/2}^{n,m}}} \right\}^2 \pi_n(m)
\le CD_n^\G(f;\nu_U^n).
\end{align*}
\end{lemma}
\begin{proof}
For each $m=0,\ldots, n-1$, note that
\begin{align*}
\lan f \ran_{\nu_{1/2}^{n,m+1}} = \int_{\Omega_{n,m+1}} f d\nu_{1/2}^{n,m+1}
=\dfrac{1}{m+1} \sum_{x\in\T_n}\int_{\Omega_{n,m+1}}\eta_x f(\eta) \nu_{1/2}^{n,m+1}(d\eta).
\end{align*}
The second equality follows from the fact that
$\sum_{x\in\T_n}\eta_x=m+1$ on $\Omega_{n,m+1}$.
Therefore the change of variable $\Omega_{n,m+1}\ni\eta=\xi^x\in\Omega_{n,m}$ yields
\begin{align*}
\lan f \ran_{\nu_{1/2}^{n,m+1}}
&=\dfrac{1}{n-m} \sum_{x\in\T_n}\int_{\Omega_{n,m}}(1-\eta_x) f(\eta^x) \nu_{1/2}^{n,m}(d\eta)\\
&=\lan f \ran_{\nu_{1/2}^{n,m}} + \sum_{x\in\T_n}\int_{\Omega_{n,m}} \dfrac{1-\eta_x}{n-m}\nabla_xf(\eta)\nu_{1/2}^{n,m}(d\eta),
\end{align*}
where $\nabla_xf(\eta)=f(\eta^x)-f(\eta)$.
It follows from this equality and the fundamental inequality
\begin{align*}
\left(\sqrt{a} - \sqrt{b}\right)^2 = \dfrac{(a-b)^2}{\left(\sqrt{a} + \sqrt{b}\right)^2}
\le \dfrac{(a-b)^2}{a+b}
\end{align*}
that
\begin{align*}
&\left\{\sqrt{\lan f \ran_{\nu_{1/2}^{n,m+1}}} - \sqrt{\lan f \ran_{\nu_{1/2}^{n,m}}} \right\}^2\\
&\le\left(\lan f \ran_{\nu_{1/2}^{n,m+1}} +\lan f \ran_{\nu_{1/2}^{n,m}}\right)^{-1}
\left(\sum_{x\in\T_n}\int_{\Omega_{n,m}} \dfrac{1-\eta_x}{n-m}\nabla_xf(\eta)\nu_{1/2}^{n,m}(d\eta)\right)^2.
\end{align*}
By the identity $\nabla_xf(\eta) = \nabla_x\sqrt{f}(\eta)(\sqrt{f(\eta^x)} + \sqrt{f(\eta)})$ and the Schwarz inequality,
the squared term in the last expression is bounded above by
\begin{align*}
&\left(\int_{\Omega_{n,m}} \sum_{x\in\T_n}\dfrac{1-\eta_x}{n-m}\nabla_x\sqrt{f}(\eta)^2\nu_{1/2}^{n,m}(d\eta) \right)\\
&\quad\times\left(\int_{\Omega_{n,m}} \sum_{x\in\T_n}\dfrac{1-\eta_x}{n-m}(\sqrt{f(\eta^x)}+\sqrt{f(\eta)})^2\nu_{1/2}^{n,m}(d\eta)\right)\\
&\le\left(\dfrac{1}{n-m}\sum_{x\in\T_n}\int_{\Omega_{n,m}}\nabla_x\sqrt{f}(\eta)^2 \nu_{1/2}^{n,m}(d\eta)\right)
\times 2\left(\lan f \ran_{\nu_{1/2}^{n,m+1}} +\lan f \ran_{\nu_{1/2}^{n,m}}\right).
\end{align*}
Collecting the above estimates, we obtain
\begin{align*}
&\sum_{m=0}^{n-1}b_n(m)\left\{\sqrt{\lan f \ran_{\nu_{1/2}^{n,m+1}}} - \sqrt{\lan f \ran_{\nu_{1/2}^{n,m}}} \right\}^2 \pi_n(m)\\
&\le 2\sum_{m=0}^{n-1} e^{n(U_0((m+1)/n) - U_0(m/n))/2}\sum_{x\in\T_n}\left(\int_{\Omega_{n,m}}\nabla_x\sqrt{f}(\eta)^2 \nu_{1/2}^{n,m}(d\eta)\right)\pi_n(m)\\
&\le C\sum_{x\in\T_n} \int_{\Omega_n}\nabla_x\sqrt{f}(\eta)^2 \nu_U^n(d\eta) \le C D_n^\G(f;\nu_U^n),
\end{align*}
for some constant $C>0$, which may change from line to line.
In the last line, we used the fact that $U_0'$ is bounded
and that the original Glauber jump rates are strictly positive.
Thus, the proof of Lemma \ref{lem3} is completed.
\end{proof}

Now we are in the position to prove Theorem \ref{thm1}.
It is a direct consequence of the previous lemmata.

\begin{proof}[Proof of Theorem \ref{thm1}]
In this proof, the exact value of a constant $C$ may change from line to line
and use the notations introduced in this section.
Let $f:\Omega_n\to\R$ be any density with respect to $\nu_U^n$.
Lemma \ref{lem1} shows that
\begin{align*}
H_n(f|\nu_U^n)
&\le Cn^2 D_n^\ex(f;\nu_U^n) + \sum_{m=0}^n\lan f \ran_{\nu_{1/2}^{n,m}}\log \lan f \ran_{\nu_{1/2}^{n,m}} \nu_U^n(\Omega_{n,m})\\
&= Cn^2 D_n^\ex(f;\nu_U^n) + \mf H_n(\mf f_n|\pi_n),\\
\end{align*}
where $\mf f_n(m):=\lan f \ran_{\nu_{1/2}^{n,m}}, m\in I_n$.
Recall the definition of $\pi_n$ and note that $\mf f_n$
is a density with respect to $\pi_n$.
Therefore, it follows from Lemmata \ref{lem2} and \ref{lem3}
that $\mf H_n(\mf f_n|\pi_n) \le C\sqrt{n}D_n^\G(f;\nu_U^n)$,
which completes the proof of Theorem \ref{thm1}.
\end{proof}

\section{Proof of Lemma \ref{lem6}}\label{sec:pflogsobolev}

In this section, we prove Lemma \ref{lem6}.
The proof of Lemma \ref{lem6} is divided into several steps
and immediately follows from Lemmata \ref{lem7}, \ref{lem8}, and \ref{lem9}.
In these lemmata, we prove the bound for $C_n^+$ since the one for $C_n^-$ is the same.

We begin with a reminder of the stationary state $\pi_n( \cdot)$. By definition,
\begin{equation*}
\pi_n(k) = \frac{1}{Z_U^n} e^{nU_0(k/n)} \binom{n}{k} \frac{1}{2^n},    
\end{equation*}
where $Z_U^n$ is a normalizing constant. 
We use Lemma B.1 from \cite{DL} repeatedly, which states that $2^nZ_U^n / n^{1/4}$ has a nontrivial limit as $n\rightarrow\infty$.
On the other hand, by Stirling's approximation, there exist constants $C_1>0$
and $C_2>0$ such that for any $n\in\N$ and $k=1,\ldots,n-1$,
\begin{equation}
C_1a_n(k) e^{-nE(k/n)} \le \binom{n}{k}\frac{1}{2^n} \le C_2a_n(k) e^{-nE(k/n)}, \label{Stirling}
\end{equation} 
where $a_n(k) = \sqrt{\frac{n}{k(n-k)}}$.
Throughout this section, we adopt the convention that $\sqrt{n-k}$ equals $1$ for $k=0$ and $k=n$, so that the notation is simplified.
We sometimes use these facts without further mention.

As mentioned above, we focus on the proof of the lemma for $C_n^+$.
By the definition of $C_n^+$, we have to show that 
there exists a constant $C>0$ such that for any $n\in\N$ and $l \in [m_n,n]$,
\begin{equation}\label{ineq7}
\sum_{k=m_n}^l \frac{1}{\pi_n(k)d_n(k)} \Psi(\pi_n([l,n])) \leq C \sqrt{n},
\end{equation}
where $\Psi(x) = -x\log x, x\ge0$. 
For this, we split the interval $[m_n,n]$ into three subintervals as
\begin{align*}
[m_n,n]=[m_n, m_n+c_2n^{3/4}] \cup [m_n+c_2n^{3/4},(3/4)n] \cup [(3/4)n,n],
\end{align*}
for some constant $c_2>0$. The choice of $c_2$ will be specified
just before Lemma \ref{lem8}.
Lemma \ref{lem7} below shows the bound \eqref{ineq7} on the interval 
$[(3/4)n,n]$, Lemma \ref{lem8} on the interval $[m_n+c_2n^{3/4},(3/4)n]$
and Lemma \ref{lem9} for $[m_n, m_n+c_2n^{3/4}]$, respectively.
We note that the exponent $3/4$ can be replaced with an arbitrary
fixed $\kappa\in(1/2, 1)$.
In the following lemmata, we give bounds on $\pi_n([l,n])$ and
$\sum_{k=m_n}^l (\pi_n(k)d_n(k))^{-1}$, respectively.

\begin{lemma}\label{lem10}
There exists a constant $C>0$ such that
for any $n\in\N$ and $l \in [(3/4)n,n]$,
\begin{align*}
\pi_n([l,n]) \leq C n^{-1/4} \frac{e^{-nW(l/n)}}{\sqrt{n-l}}.
\end{align*}
\end{lemma}

\begin{proof} 
Let $l \in [(3/4)n,n]$. By the definition of $\pi_n(\cdot)$, we have
\begin{align*}    
\pi_n([l,n]) &= \frac{1}{2^n Z_U^n} \sum_{k=l}^n \binom{n}{k} e^{nU_0(k/n)}\\
&\leq C n^{-1/4} \sum_{k=l}^n  \frac{1}{\sqrt{n-k}} e^{-nW(k/n)} \\
&= Cn^{-1/4} \frac{e^{-nW(l/n)}}{\sqrt{n-l}}  \sum_{k=l}^n \frac{\sqrt{n-l}}{\sqrt{n-k}} e^{-n \{ W(k/n) - W(l/n)\}},
\end{align*}
for some constant $C$, where we used the fact that $2^nZ_U^n/n^{1/4}$ has a limit, Stirling's approximation \eqref{Stirling} and $a_n(k)$ is bounded above by $C(n-k)^{-1/2}$.
Since $l/n \geq 3/4$, we have that there exists $c_0$ such that $W'(l/n) \geq c_0$. It follows that $W(k/n) - W(l/n) \geq c_0n^{-1}(k-l)$.
Thus we get
\begin{align*}    
\pi_n([l,n]) &\leq Cn^{-1/4} \frac{e^{-nW(l/n)}}{\sqrt{n-l}} \sum_{k=l}^n \frac{\sqrt{n-l}}{\sqrt{n-k}} e^{-c_0 (k-l)} \\
&= Cn^{-1/4} \frac{e^{-nW(l/n)}}{\sqrt{n-l}} \sum_{k=0}^{n-l} \frac{\sqrt{n-l}}{\sqrt{n-k-l}} e^{-c_0 k}.
\end{align*}
It remains to show that the previous sum is bounded above by some constant,
which is independent of $n$ and $l$.
Noting that $ n-l \leq 2(n-k-l)$ for $k \leq (n-l)/2$,
the previous sum is bounded above by
\begin{align*}
\sqrt{2} \sum_{k=0}^{(n-l)/2} e^{-c_0 k} + \sum_{k=(n-l)/2}^{n-l}   \frac{\sqrt{n-l}}{\sqrt{n-k-l}} e^{-c_0 k}.
\end{align*}
We may ignore the denominator of the second sum and
change of variables shows that the second sum is bounded above by
\begin{align*}
\sqrt{n-l} e^{-c_0(n-l)/2} \sum_{k \geq 0} e^{-c_0 k}.
\end{align*}
Since $\sqrt{x} e^{-c_0x/2}$ is bounded on $x\in[0,\infty)$,
collecting the previous estimates, we obtain the desired result.
\end{proof}

Now we give a bound for $\sum_{k=m_n}^l (\pi_n(k)d_n(k))^{-1} $.

\begin{lemma}\label{lem11}
There exists a constant $C>0$ such that
for any $n\in\N$ and $l \in [(3/4)n,n]$,
\begin{align*}
\sum_{k=m_n}^l \frac{1}{\pi_n(k)d_n(k)}  \leq C n^{-3/4}\sqrt{n-l} e^{nW(l/n)}.
\end{align*}
\end{lemma}
\begin{proof}
Let $l \in [(3/4)n,n]$.
By Stirling's approximation \eqref{Stirling}, we have to estimate
\begin{align*}
2^nZ_U^n \sum_{k=m_n}^l \frac{1}{a_n(k)d_n(k)} e^{nW(k/n)}.
\end{align*}
For this purpose, we split the interval $[m_n,l]$ into two subintervals:
$[m_n,l] = [m_n, (5/8)n] \cup [(5/8)n, l]$.

Since the function $W$ on the interval $[1/2, 5/8]$ is maximized at $5/8$,
if $n$ is large enough
the sum on the interval $[m_n,(5/8)n]$ is bounded above by
\begin{align*}
C n^{-3/4} \sum_{k=m_n}^{(5/8)n} \sqrt{n-k} e^{nW(k/n)} \leq C n^{3/4} e^{nW(5/8)} \leq C n^{-3/4} \sqrt{n-l}e^{nW(l/n)},
\end{align*}
where we used the fact that $a_n(k) \geq C/\sqrt{n-k}, d_n(k) \geq C/n$
and $l\ge(3/4)n$.
On the other hand, the sum on the interval $[(5/8)n, l]$
can be bounded as
\begin{align*}
&C n^{-3/4} \sum_{k=(5/8)n}^{l} \sqrt{n-k} e^{nW(k/n)}\\
&\quad= C n^{-3/4}\sqrt{n-l} e^{nW(l/n)} \sum_{k=(5/8)n}^{l} \frac{\sqrt{n-k}}{\sqrt{n-l}} e^{n(W(k/n)-W(l/n))},        
\end{align*}
and thus it suffices to prove that this last sum is bounded by some constant,
which is independent of $n$ and $l$. 
Since $W'$ is strictly positive on the interval $[5/8,1]$,
there exists some $c_0>0$ such that $W(l/n) - W(k/n) \geq c_0 (l-k)/n$ and thus the sum is bounded above by
\begin{align*}
\sum_{k=(5/8)n}^{l} \frac{\sqrt{n-k}}{\sqrt{n-l}} e^{-c_0(l-k)} = \sum_{k=0}^{l-(5/8)n} \frac{\sqrt{n-l+k}}{\sqrt{n-l}} e^{-c_0 k} \leq \sum_{k=0}^{l-(5/8)n} \sqrt{1+k} e^{-c_0 k},
\end{align*}
where we used the inequality $(a+k)/a \leq 1+k$ for $a \geq 1$.
This last sum is summable in $k$ and thus the lemma is proved.
\end{proof}

Combining the two previous results, we obtain the following result.

\begin{lemma}\label{lem7}
There exists a constant $C>0$ such that for any $n\in\N$ and $l \in [(3/4)n,n]$,
\begin{align*}
\sum_{k=m_n}^l \frac{1}{\pi_n(k)d_n(k)} \Psi(\pi_n([l,n])) \leq C.
\end{align*}
\end{lemma}

\begin{proof}
Let $l \in [(3/4)n,n]$.
Noting that
$$\lim_{n\rightarrow\infty} \pi_n([(3/4)n,n]) = 0,$$
by Lemma \ref{lem10} and the function $\Psi$ is increasing on $[0,e^{-1}]$,
we get
\begin{align*}
\Psi(\pi_n([l,n])) \leq \Psi \left(C n^{-1/4} \frac{e^{-nW(l/n)}}{\sqrt{n-l}}\right).
\end{align*}
Combining this with Lemma \ref{lem11}, we obtain
\begin{align*}
\sum_{k=m_n}^l \frac{1}{\pi_n(k)d_n(k)} \Psi(\pi_n([l,n]))
\leq - \frac{C}{n} \log \left( C n^{-1/4} \frac{e^{-nW(l/n)}}{\sqrt{n-l}} \right)  \leq C.
\end{align*}
In the last inequality, we used the fact that $l\ge(3/4)n$.
Thus, the proof of Lemma \ref{lem7} is completed.
\end{proof}

To obtain the bound \eqref{ineq7} for $l\in[m_n+c_2n^{3/4}, (3/4)n]$,
we proceed as in Lemma \ref{lem10}, by first estimating $\pi_n([l,n])$.

\begin{lemma}\label{lem12}
There exists a constant $K>0$ such that for large enough $c_2>0$,
for large $n\in\N$ and $l\in [m_n+c_2n^{3/4}, (3/4)n]$,
\begin{equation*}
\pi_n([l,n])  \leq K \frac{n^{3/4}}{l-m_n} e^{-n W(l/n)}.
\end{equation*}
\end{lemma}

\begin{proof}
In this proof, we need to control the dependence on $c_2$,
and thus we denote by $K$ a constant which is independent of  $c_2$,
that may change from line to line but is universal for $l$ and large $n$.
\par Let $l\in [m_n+c_2n^{3/4}, (3/4)n]$.
By Stirling's approximation and the fact that $2^nZ_U^n/n^{1/4}$
has a nontrivial limit as $n\to\infty$,
we have to estimate
\begin{align*}
\frac{1}{n^{1/4}} \sum_{k=l}^n a_n(k) e^{-nW(k/n)}.
\end{align*}
We start by breaking the interval $[l,n]$ into two intervals as
$[l,n]= [l, (5/8) n] \cup [(5/8)n,n]$.
The contribution coming from the second interval is easily calculated.
Indeed, since $l\le (3/4)n$, we obtain
\begin{align*}
\frac{1}{n^{1/4}} \sum_{k=(5/8)n}^n a_n(k) e^{-n W(k/n)} \leq n^{3/4} e^{-nW(5/8)} \le \frac{n^{3/4}}{l-m_n} e^{-n W(l/n)},
\end{align*}
if $n$ is large enough.
Here, we bounded $a_n(k)$ by $1$ and used the fact that $W$ is strictly decreasing on $[1/2,1]$. 

In the rest of this proof, we give a bound for
\begin{align*}
\frac{1}{n^{1/4}} \sum_{k=l}^{(5/8)n} a_n(k) e^{-nW(k/n)}.
\end{align*}
By the fact that $a_n(k) \leq 2n^{-1/2}$,
the previous sum is bounded above by
\begin{align*}
\frac{2}{n^{3/4}} \sum_{k=l}^{(5/8)n}e^{-n W(k/n)} 
= \frac{2}{n^{3/4}} e^{-nW(l/n)} \sum_{k=l}^{(5/8)n}e^{-n [W(k/n) -W(l/n)]}.
\end{align*}
Therefore, to complete the proof, we need to show that
\begin{align}\label{ineq2}
\sum_{k=l}^{(5/8)n}e^{-n [W(k/n) -W(l/n)]} \le K \frac{n^{3/2}}{l-m_n},
\end{align}
for some constant $K>0$.

To prove \eqref{ineq2}, we give a lower bound for $W(k/n) - W(l/n)$.
We first perform the fourth-order Taylor expansion of $W$ at $l/n$,
using that $W$ has a uniformly positive fourth derivative,
that is $W^{(4)}(\rho) \geq c_W>0, \rho\in[0,1]$.
Here and hereafter, $c_W$ represents a positive constant, which is determined by $W$,
and it may change from line to line.
Then, $W(k/n) - W(l/n)$ is bounded below by
\begin{equation*}
W'(l/n) \left(\frac{k-l}{n}\right) + \frac{W''(l/n)}{2} \left(\frac{k-l}{n}\right)^2 
+ \frac{W'''(l/n)}{6}  \left(\frac{k-l}{n}\right)^3 + \frac{c_W}{24} \left(\frac{k-l}{n}\right)^4.
\end{equation*}
We next perform Taylor expansions to obtain a bound for $W^{(i)}(l/n), i=1,2,3$.
Recall \eqref{W_properties}.
The first-order Taylor expansions for $W'$ and $W'''$ at $1/2$ show that
\begin{align*}
W'(l/n) \geq \frac{4 \theta}{\sqrt{n}} \left(\frac{l}{n} - \frac{1}{2}\right),
\quad W'''(l/n) \geq c_W\left(\frac{l}{n} - \frac{1}{2}\right).
\end{align*}
Similarly, the second-order Taylor expansion for $W''$ at $1/2$ gives us
\begin{align*}
W''(l/n) \geq W''(1/2) + \frac{c_W}{2} \left(\frac{l}{n} - \frac{1}{2}\right)^2 = \frac{4\theta}{\sqrt{n}} +\frac{c_W}{2} \left(\frac{l}{n} - \frac{1}{2}\right)^2
\ge \frac{\lambda c_W}{2} \left(\frac{l}{n} - \frac{1}{2}\right)^2,
\end{align*}
where $\lambda:=1-(8|\theta|)/(c_Wc_2^2)$.
In the last inequality, we used the fact that $l/n - 1/2 \geq c_2n^{-1/4}$. We take $c_2$ large enough so that $\lambda \geq 1/2$.
Combining these estimates, $W(k/n)-W(l/n)$ is bounded below by
\begin{align}\label{Taylor_l}
\frac{4 \theta}{ n^{5/2}}(l-m_n) (k-l) + \frac{c_W}{8n^4} (l-m_n)^2 (k-l)^2 + \frac{c_W}{24n^4} (k-l)^4.
\end{align}
We discarded the cubic term because it is nonnegative
and does not play any role in the following arguments.

Note that for the first term on the right-hand side of \eqref{Taylor_l},
we don't have control over the sign (since it is a multiple of $\theta$),
while the other two terms are positive. Thus, we distinguish between the two cases,
$\theta \geq 0$ and $\theta < 0$, together with the following observation. Denote by
\begin{align*}
A_1 := \frac{1}{ n^{5/2}}(l-m_n) (k-l),
\quad A_2 := \frac{1}{n^4} (l-m_n)^2 (k-l)^2,
\quad A_3 := \frac{1}{n^4} (k-l)^4,
\end{align*}
and note that
\begin{equation}\label{Taylor_ll}
\begin{split}
A_2 \ge A_3 \iff k \leq 2l - m_n, \quad \text{and} \quad
A_1 \ge A_2 \iff k \leq l + \frac{n^{3/2}}{l-m_n}.
\end{split}
\end{equation}

We here consider the case $\theta \geq 0$. 
We split the interval $[l, (5/8)n]$ into three subintervals as
\begin{align*}
[l, (5/8)n]
&= [l, l + \frac{n^{3/2}}{l-m_n}] \cup [l + \frac{n^{3/2}}{l-m_n}, 2l- m_n]\cup [ 2l- m_n, (5/8) n] \\
&:= I_1 \cup I_2 \cup I_3.
\end{align*}
Note that the condition $l \geq m_n + c_2 n^{3/4}$ shows that $I_2$ is not empty if $n$ is large.
According to \eqref{Taylor_l} and \eqref{Taylor_ll}, we bound $ W(k/n) - W(l/n)$ below by $A_1$ on $I_1$, by $A_2$ on $I_2$ and by $A_3$ on $I_3$,
up to multiplying constants. Thus, we have
\begin{align*}
\sum_{k=l}^{(5/8)n}e^{-n [W(k/n) -W(l/n)]} \le S_1 + S_2 + S_3,
\end{align*}
where
\begin{align*}
S_1&:= \sum_{k=l}^{l+\frac{n^{3/2}}{l-m_n}} \exp\left\{- \frac{4\theta}{n^{3/2}}(l-m_n) (k-l) \right\}, \\
S_2&:= \sum_{k=l+\frac{n^{3/2}}{l-m_n}}^{2l- m_n} \exp\left\{- \frac{c_W}{8n^{3}}(l-m_n)^2(k-l)^2 \right\}, \\
S_3&:= \sum_{k=2l - m_n}^{(5/8)n} \exp\left\{- \frac{c_W}{24n^3} (k-l)^4\right\}.
\end{align*}

We start with estimating $S_1$. $S_1$ is bounded as
\begin{align*}
S_1= \sum_{k=0}^{\frac{n^{3/2}}{l-m_n}}\exp\left\{- \frac{4\theta}{n^{3/2}}(l-m_n)k\right\} 
\leq \frac{1}{1 - \exp (- 4 \theta n^{-3/2}(l-m_n))} \leq K \frac{n^{3/2}}{l-m_n}.
\end{align*}
In the last inequality, we used the inequality $(1-e^{-x})^{-1} \leq 2/x$ for enough small $x > 0$.
For $S_2$, we rewrite $S_2$ as
\begin{equation*}
\begin{split}
S_2 &=  \sum_{k=\frac{n^{3/2}}{l-m_n}}^{l- m_n} \exp\left\{-\frac{ c_W}{8n^3} (l-m_n)^2k^2 \right\}= \sum_{k=\frac{n^{3/2}}{l-m_n}}^{l- m_n} \exp{\left\{-\frac{c_W}{8} \left(\frac{l-m_n}{n^{3/2}} k\right)^2 \right\}}\\
&\leq \sum_{k = \frac{n^{3/2}}{l-m_n}}^{\infty} \exp{\left\{-\frac{c_W}{8} \left(\frac{l-m_n}{n^{3/2}} k\right)^2 \right\}} \\
&= \frac{n^{3/2}}{l-m_n} \left[ \frac{l-m_n}{n^{3/2}} \sum_{k = \frac{n^{3/2}}{l-m_n}}^{\infty} \exp{\{-\frac{c_W}{8} \left(\frac{l-m_n}{n^{3/2}} k\right)^2 \}} \right].
\end{split}
\end{equation*}
Note that the term between brackets in the last expression
is a Riemannian sum
and the function $f(x) = \exp \left\{ -\frac{c_W}{8} x^2\right\}$ is decreasing
on the interval $[0,\infty)$.
Therefore, the sum is bounded by a Riemannian integral. Thus, we obtain
\begin{equation*}
S_2\leq  \frac{n^{3/2}}{l-m_n} \int_{0}^{\infty} e^{-\frac{c_W}{8} x^2} dx
= K \frac{n^{3/2}}{l-m_n}.
\end{equation*}
Similarly,  we rewrite $S_3$ as
\begin{equation*}
\begin{split}
S_3 &= \sum_{k=l - m_n}^{(5/8)n - l} \exp\left\{- \frac{c_W}{24n} \left(\frac{k}{n^{3/4}}\right)^4\right\} \leq \sum_{k=l - m_n}^{\infty} \exp\left\{- \frac{c_W}{24n} \left(\frac{k}{n^{3/4}}\right)^4\right\} \\
&=n^{3/4} \left[\frac{1}{n^{3/4}} \sum_{k=l - m_n}^{\infty} \exp\left\{-\frac{c_W}{24} \left(\frac{k}{n^{3/4}} \right)^4 \right\} \right].
\end{split}
\end{equation*}
Note that the term between brackets in the last expression
is a Riemannian sum
and the function $f(x) = \exp\left\{ -\frac{c_W}{24} x^4\right\}$  is decreasing
on the interval $[(l-m_n-1)n^{-3/4},\infty)$.
Therefore, the sum is bounded by a Riemannian integral. Thus, we obtain
\begin{equation*}
\begin{split}
S_3 \leq n^{3/4} \int_{\frac{l-m_n-1}{n^{3/4}}}^\infty e^{-\frac{c_W}{24} x^4} dx.
\end{split}
\end{equation*}
From the fact that $\int_y^\infty e^{-\frac{c_W}{24} x^4} dx \leq Ky^{-3} e^{-\frac{c_W}{24} y^4}$ for $y>0$, we obtain
\begin{equation*}
\begin{split}
S_3 \leq K\big(\frac{n}{l-m_n-1}\big)^3 \exp \left\{-\frac{c_W}{24} \frac{(l-m_n-1)^4}{n^3}\right\}
\end{split}
\end{equation*}
Note that $n^{-3/2}(l-m_n-1)^2 \geq c_2^2-1$ and
$x^{-1} e^{-\frac{c_W}{24}x^2} \le 1$ for $x\ge1$.
Therefore, for $c_2>0$ large enough, 
the last expression is bounded above by $Kn^{3/2}(l-m_n)^{-1}$.
Collecting all the estimates for $S_1, S_2$ and $S_3$, we obtain \eqref{ineq2}.

We next consider the case $\theta < 0$ and proceed similarly to the case $\theta \geq 0$, but we need some modifications to the proof.
In this case we split the interval $[l, (5/8)n]$ into two subintervals as
\begin{align*}
\begin{split}
[l, (5/8)n] = [l, 2l-m_n] \cup [2l-m_n,(5/8)n]
: = I_1 \cup I_2.
\end{split}
\end{align*}
Again according to \eqref{Taylor_l} and \eqref{Taylor_ll},
we bound $ W(k/n) - W(l/n)$ below by $A_1+A_2$ on $I_1$, and by $A_1+A_3$ on $I_2$,
up to multiplying constants. Thus, we have
\begin{align*}
\sum_{k=l}^{(5/8)n}e^{-n [W(k/n) -W(l/n)]} \le S_1 + S_2,
\end{align*}
where
\begin{align*}
S_1&:= \sum_{k=l}^{2l- m_n} \exp\left\{-\frac{4 \theta}{ n^{3/2}}(l-m_n) (k-l) -\frac{c_W}{8n^{3}}(l-m_n)^2(k-l)^2 \right\}, \\
S_2&:= \sum_{k=2l - m_n}^{(5/8)n} \exp\left\{-\frac{4 \theta}{ n^{3/2}}(l-m_n) (k-l) -\frac{c_W}{24n^3} (k-l)^4\right\}.
\end{align*}

We start with estimating $S_1$. $S_1$ can be rewritten as
\begin{equation*}
S_1=\sum_{k=0}^{l- m_n} \exp \left\{ -4\theta \left( \frac{l-m_n}{n^{3/2}}k\right) -\frac{c_W}{8} \left(\frac{l-m_n}{n^{3/2}} k\right)^2 \right\}.
\end{equation*}
Let $f(x) = \exp \left\{ -4\theta x - \frac{c_W}{8} x^2\right\}$. Note that $f$ is
decreasing on some interval $[x_0,\infty)$.
 Denoting $k_0 = n^{3/2} x_0(l-m_n)^{-1}$, we obtain
\begin{equation*}
\begin{split}
S_1 = &\sum_{k=0}^{k_0} \exp \left\{ -4\theta \left( \frac{l-m_n}{n^{3/2}}k\right) -\frac{c_W}{8} \left(\frac{l-m_n}{n^{3/2}} k\right)^2 \right\} \\
+&\sum_{k=k_0+1}^{l- m_n} \exp \left\{ -4\theta \left( \frac{l-m_n}{n^{3/2}}k\right) -\frac{c_W}{8} \left(\frac{l-m_n}{n^{3/2}} k\right)^2 \right\}\\
\leq 
 &\sum_{k=0}^{k_0} \exp \left\{ -4\theta \left( \frac{l-m_n}{n^{3/2}}k\right) -\frac{c_W}{8} \left(\frac{l-m_n}{n^{3/2}} k\right)^2 \right\} \\
+&\sum_{k=k_0+1}^{\infty} \exp \left\{ -4\theta \left( \frac{l-m_n}{n^{3/2}}k\right) -\frac{c_W}{8} \left(\frac{l-m_n}{n^{3/2}} k\right)^2 \right\},
\end{split}
\end{equation*}
where in the inequality we just extended the range of the second sum.
The first sum is a Riemannian sum for $f$ on the interval $[0,x_0]$ with step size $(l-m_n) n^{-3/2} \leq n^{-1/2}$. Thus, for $n$ large enough, we have
\begin{equation*}
\begin{split}
\frac{l-m_n}{n^{3/2}}&\sum_{k=0}^{k_0} \exp \left\{ -4\theta \left( \frac{l-m_n}{n^{3/2}}k\right) -\frac{c_W}{8} \left(\frac{l-m_n}{n^{3/2}} k\right)^2 \right\} \\
&\leq 2\int_0^{x_0} e^{-4\theta x-\frac{c_W}{8} x^2} dx = K.
\end{split}
\end{equation*}
On the other hand, the second sum is  also a Riemannian sum
and $f(x) = \exp \left\{ -4\theta x - \frac{c_W}{8} x^2\right\}$ is decreasing on $[x_0,\infty)$.
 Thus, it can be bounded as
\begin{equation*}
\begin{split}
\frac{l-m_n}{n^{3/2}}&\sum_{k=k_0+1}^{\infty} \exp \left\{ -4\theta \left( \frac{l-m_n}{n^{3/2}}k\right) -\frac{c_W}{8} \left(\frac{l-m_n}{n^{3/2}} k\right)^2 \right\} \\
&\leq \int_{x_0}^\infty e^{-4\theta x-\frac{c_W}{8} x^2} dx = K.\\
\end{split}
\end{equation*}
Collecting the bounds for $S_1$, we obtain
\begin{equation}\label{ineq9}
    S_1 \leq K \frac{n^{3/2}}{l-m_n}. 
\end{equation}

Similarly, $S_2$ can be rewritten as
\begin{equation*}
S_2 = n^{3/4} \left[\frac{1}{n^{3/4}} \sum_{k=l - m_n}^{(5/8)n -l}  \exp\left\{-\frac{4 \theta}{ n^{3/2}}(l-m_n) k -\frac{c_W}{24} \left(\frac{k}{n^{3/4}} \right)^4\right\} \right].
\end{equation*}
Now we use the fact $l-m_n \le k \iff -4\theta (l-m_n) \leq -4\theta k$ to obtain
\begin{equation*}
\begin{split}
S_2\leq n^{3/4} \left[\frac{1}{n^{3/4}} \sum_{k=l - m_n}^{(5/8)n -l} \exp\left\{- 4\theta \big(\frac{k}{ n^{3/4}} \big)^2 -\frac{c_W}{24} \left(\frac{k}{n^{3/4}} \right)^4\right\} \right] \\
\leq n^{3/4} \left[\frac{1}{n^{3/4}} \sum_{k=l - m_n}^{\infty} \exp\left\{- 4\theta \big(\frac{k}{ n^{3/4}} \big)^2 -\frac{c_W}{24} \left(\frac{k}{n^{3/4}} \right)^4\right\} \right].
\end{split}
\end{equation*}
Let $f(x) = \exp \left\{ -4\theta x^2 - \frac{c_W}{24} x^4\right\}$ and note that $f$ is decreasing for $x\geq x_0$, for some $x_0$.
The above sum is a Riemannian sum for $f$ on the interval $[(l-m_n)n^{-3/4},\infty)$.
Choosing $c_2 \geq x_0$, the sum is bounded by the Riemannian integral and we obtain
\begin{align*}
S_2\leq n^{3/4} \int_{\frac{l-m_n}{n^{3/4}}}^\infty e^{-4\theta x^2-\frac{c_W}{24} x^4} dx
&\leq K n^{3/4}\frac{n^{9/4}}{(l-m_n)^3} \exp \left\{-\frac{c_W}{24} \frac{(l-m_n)^4}{n^3}\right\} \\
&= K \left(\frac{n}{l-m_n}\right)^3 \exp \left\{-\frac{c_W}{24} \frac{(l-m_n)^4}{n^3}\right\},
\end{align*}
where we used the elementary inequality $\int_y^\infty e^{-ax^2-\lambda x^4} dx \leq K y^{-3} e^{-\lambda y^4}$ for $\lambda >0, a\in \R$. The last expression is again bounded by $Kn^{3/2}(l-m_n)^{-1}$.
Collecting all the estimates for $S_1$ and $S_2$, we obtain \eqref{ineq2},
which completes the proof of Lemma \ref{lem12}.
\end{proof}

We next estimate the sum $\sum_{k=m_n}^l (\pi_n(k) d_n(k))^{-1}$.
For this purpose, we start from an elementary lemma.

\begin{lemma}\label{lem14}
Let $f:\R\to\R$ be a smooth function satisfying $f(0) = f'(0) = f'''(0) = 0$ and
denote by $b := f^{(4)}(0)/24$. Suppose that $f^{(5)}(x) \geq 0$ for any $x \geq 0$.
Then the function $g(x) = x^{-2}f(x) - bx^2$ is increasing for $x \geq 0$
\end{lemma}

\begin{proof}
Note that
$$g'(x) = \frac{f'(x)x - 2f(x) -2bx^4}{x^3},$$
and thus we need to show that
$$h(x) := f'(x)x - 2f(x)-2bx^4 \ge 0,$$ 
for any  $x\geq 0$. A simple calculation shows that
the first three derivatives of $h$ at $0$ vanish and $h^{(4)}(x) \geq 0$ for $x \geq 0$.
 The conclusion follows from the fourth-order Taylor expansion for $h$ at $0$.
\end{proof}

We are now in position to estimate $\sum_{k=m_n}^l (\pi_n(k) d_n(k))^{-1}$.

\begin{lemma}\label{lem13}
There exists a constant $C>0$ such that
for any $n\in\N$ and $l\in [m_n+c_2n^{3/4}, (3/4)n]$,
\begin{align}\label{ineq6}
\sum_{k=m_n}^l \frac{1}{\pi_n(k) d_n(k)} \leq  C \frac{n^{11/4}}{(l-m_n)^3}e^{nW(l/n)}.
\end{align}
\end{lemma}

\begin{proof}
Let $l\in [m_n+c_2n^{3/4}, (3/4)n]$.
By Stirling's approximation \eqref{Stirling}, we have to estimate
\begin{equation*}
2^nZ_U^n \sum_{k=m_n}^l \frac{1}{a_n(k)d_n(k)} e^{nW(k/n)}.
\end{equation*}
Using the fact that $a_n(k) \geq C n^{-1/2}$, $d_n(k) \geq C n$ and $2^n Z_U^n \leq C n^{1/4}$, we obtain
\begin{align}\label{ineq4}
2^nZ_U^n \sum_{k=m_n}^l \frac{1}{a_n(k)d_n(k)} e^{nW(k/n)} \leq C n^{-1/4} \sum_{k=m_n}^l e^{nW(k/n)}.
\end{align}

Note that $W(1/2) = W'(1/2) = W'''(1/2) = 0$.
Moreover, since $W^{(5)}(1/2) = 0$ and $W^{(6)}(\cdot+1/2) \ge 0$, $W^{(5)}(x) \geq 0$ for any $x \in [1/2, 3/4]$.
Therefore, we apply Lemma \ref{lem14} to the function $W$.
Then for any $l\in [m_n+c_2n^{3/4}, (3/4)n]$ and $k\in[m_n, l]$, we obtain
\begin{equation*}
\frac{W(k/n)}{(\frac{k}{n} -\frac{m_n}{n})^2} -b(\frac{k}{n} -\frac{m_n}{n})^2 \leq \frac{W(l/n)}{(\frac{l}{n} -\frac{m_n}{n})^2} - b(\frac{l}{n} -\frac{m_n}{n})^2,
\end{equation*}
where $b:=W^{(4)}(1/2)/24$.
On the other hand, by the sixth-order Taylor expansion for $W$ at $m_n/n$, there exists some $R_l\in\R$ such that
\begin{equation*}
W(l/n) =  \frac{2\theta}{\sqrt{n}} (\frac{l}{n} -\frac{m_n}{n})^2 + b (\frac{l}{n} -\frac{m_n}{n})^4 + R_l (\frac{l}{n} -\frac{m_n}{n})^6.
\end{equation*}
Combining these two estimates, we obtain
\begin{align}
W(k/n) &\leq \frac{2\theta}{\sqrt{n}} (\frac{k}{n} -\frac{m_n}{n})^2 + b(\frac{k}{n} -\frac{m_n}{n})^4 + R_l (\frac{l-m_n}{n})^4(\frac{k-m_n}{n})^2 \notag\\
&\leq \frac{2\theta}{\sqrt{n}} (\frac{k}{n} -\frac{m_n}{n})^2 + b(\frac{k}{n} -\frac{m_n}{n})^4 + R_l (\frac{l-m_n}{n})^6. \label{ineq3}
\end{align}

It follows from \eqref{ineq3} that
\begin{align*}
\sum_{k=m_n}^l e^{nW(k/n)} 
&\leq \exp\left\{nR_l(\frac{l-m_n}{n})^6\right\}\sum_{k=m_n}^l \exp \left\{ 2\theta \sqrt{n} (\frac{k}{n} -\frac{m_n}{n})^2  + nb(\frac{k}{n} -\frac{m_n}{n})^4  \right\}  \\
&= \exp\left\{nR_l\frac{l-m_n}{n}^6\right\} n^{3/4} \left[ \dfrac{1}{n^{3/4}}\sum_{k=0}^{l-m_n} \exp\left\{2\theta (\frac{k}{n^{3/4}})^2  + b (\frac{k}{n^{3/4}})^4\right\} \right].
\end{align*}
The last expression can be bounded as we did for the bound \eqref{ineq9}.
Indeed, let $f(x) = \exp \left\{ 2\theta x^2 + b x^4\right\}$ and note that
$f$ is increasing for $x \geq x_0$, for some $x_0$.
Let $k_0 = n^{3/4}x_0$ and write
\begin{equation*}
\begin{split}
&\sum_{k=0}^{l-m_n} \exp\left\{2\theta (\frac{k}{n^{3/4}})^2  + b (\frac{k}{n^{3/4}})^4\right\} \\
= &\sum_{k=0}^{k_0} \exp\left\{2\theta (\frac{k}{n^{3/4}})^2  + b (\frac{k}{n^{3/4}})^4\right\} + \sum_{k=k_0+1}^{l-m_n} \exp\left\{2\theta (\frac{k}{n^{3/4}})^2  + b (\frac{k}{n^{3/4}})^4\right\}.
\end{split}
\end{equation*}
The first sum is a Riemannian sum for $f$ on the interval $[0,x_0]$
and in this way is bounded by the Riemannian integral multiplied by some constant.
On the other hand, the second sum is also a Riemannian sum
for the increasing function $f$ on the interval $[x_0+n^{-3/4}, (l-m_n +1)n^{-3/4}]$.
Thus, it is bounded by the Riemannian integral and we obtain
\begin{equation*}
\begin{split}
\sum_{k=0}^{l-m_n} \exp\left\{2\theta (\frac{k}{n^{3/4}})^2  + b (\frac{k}{n^{3/4}})^4\right\} &
\leq C\int_0^{x_0} e^{2\theta x^2 + bx^4} dx + \int_{x_0}^{\frac{l-m_n+1}{n^{3/4}}} e^{2\theta x^2 + bx^4} dx \\
&\leq C\int_{0}^{\frac{l-m_n+1}{n^{3/4}}} e^{2\theta x^2 + bx^4} dx.
\end{split}
\end{equation*}
Combining these estimates gives us
\begin{align}\label{ineq5}
\sum_{k=m_n}^l e^{nW(k/n)}  \leq C\exp\left\{nR_l(\frac{l-m_n}{n})^6\right\}  n^{3/4} \int_{0}^{\frac{l-m_n+1}{n^{3/4}}} e^{2\theta x^2 + bx^4} dx.
\end{align}
Therefore, by \eqref{ineq4}, \eqref{ineq5} and the elementary inequality
$$
\int_0^y e^{2\theta  x^2 + bx^4} dx \leq C y^{-3} e^{2\theta y^2 + b y^4}, \quad y\ge1,
$$
we obtain \eqref{ineq6}.
\end{proof}

Now we fix $c_2>0$ large enough so that $K c_2^{-1} < e^{-1}$ is satisfied,
where $K$ is given by \ref{lem7}.
Combining Lemma \ref{lem12} and Lemma \ref{lem13}, we obtain the following result.

\begin{lemma}\label{lem8}
There exists a constant $C>0$ such that
for any $n\in\N$ and $l\in [m_n+c_2n^{3/4}, (3/4)n]$,
\begin{equation*}
\sum_{k=m_n}^l \frac{1}{\pi_n(k)d_n(k)} \Psi(\pi_n([l,n])) \leq C \sqrt{n}.
\end{equation*}
\end{lemma}

\begin{proof}
Let $l\in [m_n+c_2n^{3/4}, (3/4)n]$.
By Lemma \ref{lem12},
\begin{equation*}
\pi_n([l,n]) \leq K \frac{n^{3/4}}{l-m_n} e^{-n W(l/n)} \leq K c_2^{-1}
\end{equation*}
and it follows from the choice of $c_2$ that $\pi_n([l,n]) < e^{-1}$.
Since the function $\Psi$ is increasing on $[0, e^{-1}]$, we get
\begin{equation*}
\Psi(\pi_n([l,n])) \leq 
 C \frac{n^{3/4}}{l-m_n} e^{-n W(l/n)} \log\left( \frac{l-m_n}{Cn^{3/4}} e^{n W(l/n)}\right).
\end{equation*}
By \eqref{lem13}, we obtain
\begin{align*}
\sum_{k=m_n}^l \frac{1}{\pi_n(k)d_n(k)} \Psi(\pi_n([l,n])) \leq C \dfrac{n^{7/2}}{(l-m_n)^4}\log\left( \frac{l-m_n}{Cn^{3/4}} e^{n W(l/n)}\right).
\end{align*}
Note that there exists a constant $C>0$ such that for any $n\in\N$ and $l\in[m_n+c_2n^{3/4}, (3/4)n]$,
$$
W(l/n) \leq  \frac{2\theta}{\sqrt{n}} (\frac{l}{n} -\frac{m_n}{n})^2 + C (\frac{l}{n} -\frac{m_n}{n})^4.
$$
Thus, letting $x = n^{-3/4}(l-m_n)$ and noting that $x \geq c_2$,
the right-hand side of the penultimate display is bounded above by
\begin{equation*}
C n^{1/2} \sup_{x\ge c_2}\left\{x^{-4} \left( \dfrac{\log x}{C} + 2\theta x^2+Cx^4 \right) \right\} = C n^{1/2},
\end{equation*}
which completes the proof of Lemma \ref{lem8}.
\end{proof}

We now focus on the interval $[m_n, m_n+ c_2n^{3/4}]$.
Although the proofs of the following lemmata are similar to the previous ones,
we give them for the reader's convenience.

\begin{lemma}\label{lem15}
There exists a constant $C>0$ such that
for any $n\in\N$ and $l\in [m_n, m_n+c_2n^{3/4}]$,
\begin{equation*}
\pi_n([l,n])  \leq C.
\end{equation*}
\end{lemma}

\begin{proof}
Let $l\in [m_n, m_n+c_2n^{3/4}]$.
We have to estimate $n^{-1/4} \sum_{k=l}^n a_n(k) e^{-nW(k/n)}$.
We split the interval $[l,n]$ into two subintervals as
$$[l,n] = [l, m_n+ (1/4)n] \cup  [m_n + (1/4)n, n]=:I_1+I_2.$$
The contribution coming from the interval $I_2$ is easily calculated.
Indeed, the sum over the interval $I_2$ is bounded as
\begin{equation*}
\frac{1}{n^{1/4}} \sum_{k=(3/4)n}^n a_n(k) e^{-n W(k/n)} \leq Cn^{3/4} e^{-nW(3/4)} \leq C,
\end{equation*}
if $n$ is large enough.

Now we focus on the sum over the interval $I_1$.
We use the bound $a_n(k)\le Cn^{-1/2}$ to obtain
\begin{equation*}
\frac{1}{n^{1/4}} \sum_{k=l}^{(3/4)n} a_n(k) e^{-n W(k/n)} \leq \frac{C}{n^{3/4}} \sum_{k=l}^{(3/4)n}e^{-n W(k/n)}.
\end{equation*}
Then it follows from the fourth-order Taylor expansion that
$$W(k/n) \geq \frac{2\theta}{\sqrt{n}} \big(\frac{k-m_n}{n}\big)^2 + \frac{c_W}{24} \big(\frac{k-m_n}{n}\big)^4,$$
that
\begin{align*}
\frac{C}{n^{3/4}} \sum_{k=l}^{(3/4)n}e^{-n W(k/n)}
&\leq \frac{C}{n^{3/4}} \sum_{k=l}^{(3/4)n} \exp \left\{-2\theta\sqrt{n}\big(\frac{k-m_n}{n}\big)^2 - \frac{c_Wn}{24} \big(\frac{k-m_n}{n}\big)^4 \right\} \\
&= \frac{C}{n^{3/4}} \sum_{k=l-m_n}^{(1/4)n} \exp \left\{-2\theta\big(\frac{k}{n^{{3/4}}}\big)^2 - \frac{c_W}{24} \big(\frac{k}{n^{3/4}}\big)^4 \right\}\\
&\le \frac{C}{n^{3/4}} \sum_{k=0}^{\infty} \exp \left\{-2\theta\big(\frac{k}{n^{{3/4}}}\big)^2 - \frac{c_W}{24} \big(\frac{k}{n^{3/4}}\big)^4 \right\}.
\end{align*}
The last expression is a Riemann sum, thus, we have
\begin{equation*}
\leq C \int_0^\infty e^{-2\theta x^2 - \frac{c_W}{24} x^4} dx = C.
\end{equation*}
This concludes the proof of Lemma \ref{lem15}.
\end{proof}

We now estimate the sum $\sum_{k=m_n}^l (\pi_n(k) d_n(k))^{-1}$.

\begin{lemma}\label{lem16}
There exists a constant $C>0$ such that
for any $n\in\N$ and $l\in [m_n, m_n+c_2n^{3/4}]$,
\begin{equation*}
\sum_{k=m_n}^l \frac{1}{\pi_n(k) d_n(k)} \leq C \sqrt{n}.
\end{equation*}
\end{lemma}

\begin{proof}
Let $l\in [m_n, m_n+c_2n^{3/4}]$.
Similar to the estimates as we did before, we have to estimate
$n^{-1/4} \sum_{k=m_n}^l e^{nW(k/n)}$.
We use the fourth-order Taylor expansion to obtain
\begin{equation*}
W(k/n) \leq \frac{2\theta}{\sqrt{n}} \big(\frac{k-m_n}{n}\big)^2 + C \big(\frac{k-m_n}{n}\big)^4.
\end{equation*}
Thus, we obtain
\begin{align*}
\sum_{k=m_n}^l e^{nW(k/n)}
&\leq n^{3/4} \left[\frac{1}{n^{3/4}} \sum_{k=0}^{l-m_n} \exp \left\{ 2\theta \big(\frac{k}{n^{3/4}}\big)^2 + C \big(\frac{k}{n^{3/4}}\big)^4 \right\}\right] \\
&\le n^{3/4} \left[\frac{1}{n^{3/4}} \sum_{k=0}^{c_2n^{3/4}} \exp \left\{ 2\theta \big(\frac{k}{n^{3/4}}\big)^2 +C \big(\frac{k}{n^{3/4}}\big)^4 \right\}\right],
\end{align*}
since $l-m_n \le n^{3/4}$. Since the term between the brackets is a Riemann sum,
the above expression is bounded by
\begin{align*}
C n^{3/4} \int_0^{c_2} e^{2\theta x^2+ C x^4} dx = Cn^{3/4}.
\end{align*}
Therefore, collecting the previous bounds gives us
\begin{equation*}
\sum_{k=m_n}^l \frac{1}{\pi_n(k) d_n(k)} \leq C n^{-1/4}\sum_{k=m_n}^l e^{nW(k/n)} \leq C n^{1/2},
\end{equation*}
which is the desired bound.
\end{proof}

The following bound immediately follows from Lemmata \ref{lem15} and \ref{lem16}.

\begin{lemma}\label{lem9}
There exists a constant $C>0$ such that
for any $n\in\N$ and $l\in [m_n, m_n+c_2n^{3/4}]$,
\begin{equation*}
\sum_{k=m_n}^l \frac{1}{\pi_n(k)d_n(k)} \Psi(\pi_n([l,n])) \leq C \sqrt{n}.
\end{equation*}
\end{lemma}
\begin{proof}
Let $l\in [m_n, m_n+c_2n^{3/4}]$.
Since the function $\Psi$ is bounded on bounded intervals,
by Lemma \ref{lem15} we obtain
\begin{equation*}
\sum_{k=m_n}^l \frac{1}{\pi_n(k)d_n(k)} \Psi(\pi_n([l,n])) \leq C\sum_{k=m_n}^l \frac{1}{\pi_n(k)d_n(k)}.
\end{equation*}
Then the conclusion immediately follows from Lemma \ref{lem16}.
\end{proof}


In Lemma \ref{lem2}, we proved that the logarithmic Sobolev constant
$\lambda_n$ defined in \eqref{ls1} is at least of order $n^{-1/2}$.
We also claim that this order is optimal in $n$, although we do not
use this fact in the rest of this paper.

\begin{proposition}\label{prop1}
There exists a constant $C>0$ such that for any $n\in\N$,
$$\lambda_n\le Cn^{-1/2}.$$
\end{proposition}

\begin{proof}
The upper bound on $\lambda_n$ follows from the comparison between $\lambda_n$ and
the the spectral gap $\gamma_n$ defined by
\begin{align}\label{ineq8}
\gamma_n:=\inf\left\{\dfrac{\mf D_n(\mf f;\pi_n)}{\mf V_n(\mf f|\pi_n)} \,\middle|\, \mf f:I_n\to\R, \mf V(\mf f|\pi_n)\neq0\right\},
\end{align}
where
\begin{align*}
\mf D_n(\mf f;\pi_n)&:=\dfrac12\sum_{m=0}^{n-1}b_n(m)\left\{ \mf f(m+1) - \mf f(m) \right\}^2 \pi_n(m). \\
\mf V_n(\mf f|\pi_n)&:=E_{\pi_n}\left[ (\mf f - E_{\pi_n}[\mf f])^2 \right].
\end{align*}
It is well-known that $2\lambda_n\le \gamma_n$, cf. \cite[Lemma 3.1]{DSC}.
Therefore, to obtain the conclusion it is enough to show that $\gamma_n\le Cm^{-1/2}$
for some $C>0$.

To obtain the bound $\gamma_n\le Cn^{-1/2}$, we choose a test function in \eqref{ineq8}
as $\mf f_n (m):= n^{-3/4}m, m\in I_n$. Since for any $m\in I_n, b_n(m)\le Cn$ for some
$C>0$, which is independent of $n$, $\mf D_n(\mf f_n;\pi_n)\le Cn^{-1/2}$. On the other hand,
it is not difficult to see that
\begin{align*}
\lim_{n\to\infty} \mf V_n(\mf f_n|\pi_n) = \int_\R y^2 \alpha^*(dy).
\end{align*}
Thus, the proof is complete.
\end{proof}

\smallskip\noindent{\bf Acknowledgments.}  The authors would like to
thank Benoit Dagallier for providing useful comments at an earlier
stage of this work, especially about the proof of Proposition
\ref{prop1} Part of this work was done during K. Tsunoda's visit to
IMPA in 2025.  He would like to thank IMPA for the numerous supports
and warm hospitality during his visit.  C. Landim has been partially
supported by FAPERJ CNE E-26/201.117/2021, and by CNPq Bolsa de
Produtividade em Pesquisa PQ 305779/2022-2.  K. Tsunoda has been
partially supported by JSPS KAKENHI, Grant-in-Aid for Early-Career
Scientists 22K13929.

\end{document}